\documentclass[12pt,a4paper]{article}
\usepackage{mathrsfs}
\usepackage{epsfig, graphicx}
\usepackage{latexsym,amsfonts,amsbsy,amssymb}
\usepackage{amsmath,amsthm}
\usepackage{color}
\usepackage[colorlinks, citecolor=blue]{hyperref}

\usepackage{color}

\makeatletter 

\@addtoreset{equation}{section}
\makeatother 
\allowdisplaybreaks 

\newtheorem{theorem}{Theorem}[section]
\newtheorem{lemma}{Lemma}[section]
\newtheorem{corollary}{Corollary}[section]
\newtheorem{remark}{Remark}[section]
\newtheorem{algorithm}{Algorithm}[section]

\newtheorem{proposition}{Proposition}[section]

\begin{document}

\title{A Multilevel Newton Iteration Method for Eigenvalue Problems\thanks{This work is
supported in part by National Science Foundations
of China (NSFC 91330202, 11001259, 11371026, 11201501, 11031006, 2011CB309703, 11171251),
the National Basic Research Program (2012CB955804), the Major Research Plan of the National Natural Science Foundation of China (91430108), the Major Program of Tianjin University of Finance and Economics (ZD1302) and
the National Center for Mathematics and Interdisciplinary Science,
CAS and the President Foundation of AMSS-CAS.}}
\author{
Yunhui He\footnote{Department of Mathematics and Statistics,
Memorial University of Newfoundland, St. John's, NL, Canada ({\tt  yh6171@mun.ca})}, \ \
Yu Li\footnote{Research Center for Mathematics and Economics,
 Tianjin University of Finance and Economics, Tianjin 300222, China
({\tt liyu@lsec.cc.ac.cn})} \ \ and\ \
Hehu Xie\thanks{LSEC, NCMIS, Academy of Mathematics and Systems Science,
Chinese Academy of Sciences, Beijing 100190, China
({\tt hhxie@lsec.cc.ac.cn})}
}
\date{}
\maketitle
\begin{abstract}
We propose a new type of multilevel  method for solving eigenvalue problems
based on Newton iteration. With the proposed iteration method, solving eigenvalue problem on the finest
finite element space is replaced by solving a small scale eigenvalue problem in a coarse space and
solving a series of augmented linear problems, derived by Newton step in the corresponding series of finite element spaces.
 This iteration scheme improves overall efficiency of the finite element method for solving eigenvalue problems.
  Finally,  some numerical examples are provided to validate the efficiency of the proposed numerical scheme.
\end{abstract}
\section{Introduction}
 The original purpose of Newton's method is to seek the root of an equation. With a suitable initial guess,
 Newton iteration is usually convergent. Furthermore, the convergence is at
 least quadratic in a neighborhood of a simple root. So Newton's iteration is an extremely powerful technique in numerical methods.
 Nowadays Newton's method is widely applied to minimization and maximization problems, multiplicative inverses
 of numbers and power series, solving transcendent equations, complex functions, nonlinear systems of equations.

Our work is to design a Newton's method to solve PDE eigenvalue problems.
Taking advantage of the rapid convergence of Newton's method,
  we design a Newton's method to solve eigenvalue problems, treating eigenvalue problem as a nonlinear equation.
 Some works \cite{Golub2012Matrix,Saad1992Numerical,Sleijpen2000Jacobi,Sleijpen2006Jacobi} had exploited the
 Newton's method for eigenvalue problems.
 The Newton's method (see \cite{Golub2012Matrix}) is based on
 an approximate eigenpair $(\lambda_{0},x_{0})$ and wishes to determine $\delta \lambda$ and $\delta x$ so
  that $(\lambda_0+\delta \lambda,x_0+\delta x)$ is an improved approximation of the exact eigenpair.

In recent decades, the study of solving large scale eigenvalue problems, arising from modern science and
engineering society, has become one of the major focuses of numerical analysts and engineers. However,
it is always a difficult task to solve high-dimensional eigenvalue problems which come from physical and chemical sciences.
About the solution of eigenvalue problems, \cite{Brandt1983Multigrid,Hackbusch2013Multi,Hackbusch1979computation,Lin2011Observation,Lin2011type,Shaidurov2013Multigrid}
and the references cited therein give some types of multilevel or multigrid schemes.

The aim of this paper is to present a type of multilevel iteration scheme based on
Newton's method for eigenvalue problems.
The standard Galerkin finite element method for eigenvalue problems
has been extensively investigated, e.g. Babu\v{s}ka and Osborn
\cite{Babuska1989Finite,Babuska1991Eigenvalue}, Chatelin \cite{Chatelin1983Spectral} and
references cited therein. Here we adopt some basic results in these
papers for our analysis.
The corresponding error and complexity discussion of the proposed iteration
scheme for the eigenvalue problem will be analyzed. Based
on the analysis, the new method can obtain optimal errors with an optimal
computational work when we can solve the associated augmented linear
problems with the optimal complexity.  Although the Newton's method is sensitive
to initial guess, we use multilevel
 technique to overcome this difficulty. Since it is easy to find a good approximation in
 the coarse grid, which provides a good initial guess for the fine grid, the Newton type
 iteration method is reasonable.  According to the theory for mixed finite element method,
 we prove the existence and the uniqueness of the solution to the proposed scheme.


This paper is organized at follows. In Section 2, we introduce the
finite element method for the eigenvalue problem and give the corresponding basic
error estimates. A type of one Newton iteration step
 is presented and the error estimates of the proposed scheme are analyzed in Section 3.
In Section 4, we propose a type of multilevel iteration scheme for multi eigenvalues solving.
The computational work estimate of the multilevel iteration method is discussed
in Section 5. In Section 6, two numerical examples are presented to validate our
theoretical analysis. Some concluding remarks are provided in the final section.
\section{Finite element method for eigenvalue problems}
In this section, we introduce some notation and error estimates of
the finite element approximation for the eigenvalue problem.
The letter $C$ (with or without subscripts) denotes a generic
positive constant which may be different at its different occurrences through
the paper. For convenience, the symbols $\lesssim$, $\gtrsim$ and $\approx$
will be used in this paper. These $x_1\lesssim y_1, x_2\gtrsim y_2$
and $x_3\approx y_3$, mean that $x_1\leq C_1y_1$, $x_2 \geq c_2y_2$
and $c_3x_3\leq y_3\leq C_3x_3$ for some constants $C_1, c_2, c_3$
and $C_3$ that are independent of mesh sizes (see, e.g., \cite{Xu1992Iterative}).


In our methodology description, we are concerned with the following
model problem:

\noindent Find $(\lambda, u )\in \mathbb{R}\times V$ such that $b(u,u)=1$ and
\begin{eqnarray}\label{weak_problem}
a(u,v)&=&\lambda b(u,v),\quad \forall v\in V,
\end{eqnarray}
where $V:=H_0^1(\Omega)$, $a(\cdot, \cdot)$ and $b(\cdot,\cdot)$ are bilinear forms
defined by
\begin{eqnarray*}
a(u,v)=\int_{\Omega}\nabla u\nabla vd\Omega, \ \ \ \
b(u,v)=\int_{\Omega}uv d\Omega.
\end{eqnarray*}
In this paper, based on these two bilinear forms, we define the norms $\|\cdot\|_a$ and
$\|\cdot\|_b$ as follows
\begin{eqnarray*}
\|v\|_a^2=a(v,v),\ \ \ \|v\|_b^2=b(v,v).
\end{eqnarray*}
It is well known that the norm $\|\cdot\|_a$ is a norm in the space $V$ and $\|\cdot\|_b$
is a norm in the space $L^2(\Omega)$.

For the eigenvalue $\lambda$, there exists the following Rayleigh
quotient expression (see, e.g., \cite{Babuska1989Finite,Babuska1991Eigenvalue,Xu2001two})
\begin{eqnarray*}\label{Rayleigh_quotient}
 \lambda=\frac{a(u,u)}{b(u,u)}.
\end{eqnarray*}
From \cite{Babuska1991Eigenvalue,Chatelin1983Spectral}, we know the eigenvalue problem
\eqref{weak_problem} has an eigenvalue sequence $\{\lambda_j \}:$
$$0\leq\lambda_1\leq \lambda_2\leq\cdots\leq\lambda_k\leq\cdots,\ \ \
\lim_{k\rightarrow\infty}\lambda_k=\infty,$$ and the associated
eigenfunctions
$$u_1, u_2, \cdots, u_k, \cdots, $$
where $b(u_i,u_j)=\delta_{ij}$, $\delta_{ij}$ is Kronecker notation. In the sequence $\{\lambda_j\}$, the
$\lambda_j$ are repeated according to their geometric multiplicity.
In order to give the error estimates, let $M(\lambda_i)$ denote the eigenfunction space corresponding to the
eigenvalue $\lambda_i$ which is defined by
\begin{eqnarray*}
M(\lambda_i)&=&\big\{w\in V: w\ \mbox{ is\ an\ eigenfunction\ of\
\eqref{weak_problem}}\mbox{  corresponding\ to}\ \lambda_i\big\}.
\end{eqnarray*}

Now, let us define the finite element approximations of the problem \eqref{weak_problem}. First we
generate a shape-regular decomposition of the computing domain
$\Omega \subset \mathbb{R}^d$ $(d = 2,3)$ into triangles or rectangles for $d = 2$
(tetrahedrons or hexahedrons for $d = 3$). The diameter of a cell $K\in \mathcal{T}_h$ is
denoted by $h_K$. The mesh diameter $h$ describes the maximum diameter of all cells
$K \in \mathcal{T}_h$ . Based on the mesh $\mathcal{T}_h$, we can construct the
linear finite element space denoted by $V_h\subset V$.
We assume that the finite element space $V_h$ satisfies the following assumption:\\
For any $w \in V$
\begin{eqnarray}\label{Approximation_Property}
\lim_{h\rightarrow0}\inf_{v_h\in V_h}\|w-v_h\|_a = 0.
\end{eqnarray}

The finite element approximation for \eqref{weak_problem} is defined as follows:
Find $(\bar\lambda_h,\bar u_h) \in \mathbb{R}\times V_h$ such that
$b(\bar u_h,\bar u_h)=1$ and
\begin{eqnarray}\label{weak_problem_Discrete}
a(\bar u_h,v_h)=\bar\lambda_h b(\bar u_h,v_h), \quad \forall v_h\in V_h.
\end{eqnarray}

From \eqref{weak_problem_Discrete}, we know the following
Rayleigh quotient expression for $\bar{\lambda}_h$ holds
(see, e.g., \cite{Babuska1989Finite,Babuska1991Eigenvalue,Xu2001two})
\begin{eqnarray*}\label{eigenvalue_Rayleigh}
\bar{\lambda}_h &=&\frac{a(\bar{u}_h,\bar{u}_h)}{b(\bar{u}_h,\bar{u}_h)}.
\end{eqnarray*}
Similarly, we know from \cite{Babuska1991Eigenvalue,Chatelin1983Spectral} the eigenvalue
problem \eqref{weak_problem_Discrete} has eigenvalues
$$0<\bar{\lambda}_{1,h}\leq \bar{\lambda}_{2,h}\leq\cdots
\leq \bar{\lambda}_{k,h}\leq\cdots\leq \bar{\lambda}_{N_h,h},$$
and the corresponding eigenfunctions
$$\bar{u}_{1,h}, \bar{u}_{2,h},\cdots, \bar{u}_{k,h}, \cdots, \bar{u}_{N_h,h},$$
where $b(\bar{u}_{i,h},\bar{u}_{j,h})=\delta_{ij}, 1\leq i,j\leq N_h$ ($N_h$ is
the dimension of the finite element space $V_h$).

From the minimum-maximum principle (see, e.g., \cite{Babuska1989Finite,Babuska1991Eigenvalue}),
the following upper bound result holds
$${\lambda}_i\leq \bar{\lambda}_{i,h}, \ \ \ i=1,2,\cdots, N_h.$$
Similarly, let $M_h(\lambda_i)$ denote the approximate eigenfunction space corresponding to the
eigenvalue $\lambda_i$ which is defined by
\begin{eqnarray*}
M_h(\lambda_i)&=&\big\{w_h\in V_h: w_h\ \mbox{ is\ an\ eigenfunction\ of\
\eqref{weak_problem_Discrete}} \nonumber\\
&&\ \ \ \ \ \ \  \mbox{  corresponding\ to}\ \lambda_i\big\}. 
\end{eqnarray*}
From \cite{Babuska1989Finite,Babuska1991Eigenvalue}, each eigenvalue $\bar{\lambda}_{i,h}$ can be defined as follows
\begin{eqnarray}\label{Eigenvalue_h_definition}
\bar{\lambda}_{i,h}&=&\inf_{v_h\in V_h\atop v_h\perp M_h(\lambda_j)\ \mbox{ for}\ \lambda_j<\lambda_i}
\frac{a(v_h,v_h)}{b(v_h,v_h)}.
\end{eqnarray}

In order to give the error estimate result for the eigenvalue problems by the finite element method,
we define
\begin{eqnarray}\label{def:delta}
\delta_h(\lambda_i)=\sup_{w\in M(\lambda_i),\|w\|_a=1}\inf_{v_h\in
V_h}\|w-v_h\|_a,
\end{eqnarray}
and $\eta_a(h)$ as
\begin{eqnarray}\label{def:eta}
\eta_a(h)=\sup_{f\in V, \|f\|_b=1}\inf_{v_h\in V_h}\|T f-v_h\|_a,
\end{eqnarray}
where the operator $T: V'\rightarrow V$ is defined as
\begin{eqnarray*}
a(Tf,v)&=&b(f,v),\ \ \ \forall f\in V', \ \ \forall v\in V.
\end{eqnarray*}

There exist the following error estimates for the eigenpair approximations by finite element method.
\begin{proposition}(\cite[Lemma 3.7, (3.29b)]{Babuska1989Finite}, \cite[ P. 699]{Babuska1991Eigenvalue} and
\cite{Chatelin1983Spectral})\label{Error_estimate_Proposition}

\noindent(i) For any eigenfunction approximation $\bar u_{i,h}$ of
\eqref{weak_problem_Discrete} $(i = 1, 2, \cdots, N_h)$, there is an
eigenfunction $u_i$ of \eqref{weak_problem} corresponding to
$\lambda_i$ such that $\|u_i\|_b = 1$ and
\begin{eqnarray}\label{Eigenfunction_Error}
\|u_i-\bar{u}_{i,h}\|_a \leq C\delta_h(\lambda_i).
\end{eqnarray}
Furthermore,
\begin{eqnarray}\label{Eigenfunction_Error_Nagative}
\|u_i- \bar{u}_{i,h}\|_{b} \leq C\eta_a(h)\|u_i - \bar u_{i,h}\|_a.
\end{eqnarray}
(ii) For each eigenvalue, we have
\begin{eqnarray}\label{Eigenvalue_Error}
\lambda_i \leq \bar{\lambda}_{i,h}\leq \lambda_i +
C\delta_h^2(\lambda_i).
\end{eqnarray}
Here and hereafter $C$ is some constant depending on $\lambda_i$ but independent of  the mesh size $h$.
\end{proposition}

\section{A Newton iteration method for eigenvalue problem}
The aim of this section is to present a type of one Newton iteration step
to improve the accuracy of the given eigenpair approximations.
This iteration method only contains solving augmented linear problems
in a finer finite element space.
Here we only state the numerical method for the first and simple
eigenvalue. In the next section, we will show the case of multi eigenvalues.

 For the analysis in this paper, we introduce the error expansion of the
eigenvalue by the Rayleigh quotient formula which comes from
\cite{Babuska1989Finite,Babuska1991Eigenvalue,Lin1996construction,Xu2001two}.
\begin{lemma}[\cite{Babuska1989Finite,Babuska1991Eigenvalue,Lin1996construction,Xu2001two}]\label{Rayleigh_Quotient_error_theorem}
Assume $(\bar\lambda_h,\bar u_h)$ is a true solution of the eigenvalue problem
\eqref{weak_problem_Discrete} and  $0\neq \psi_h\in V_h$. Let us define
\begin{eqnarray*}\label{rayleighw}
\widehat{\lambda}_h=\frac{a(\psi_h,\psi_h)}{b(\psi_h,\psi_h)}.
\end{eqnarray*}
Then we have
\begin{eqnarray*}\label{rayexpan}
\widehat{\lambda}_h-\bar\lambda_h
&=&\frac{a(\bar u_h-\psi_h,\bar u_h-\psi_h)}{b(\psi_h,\psi_h)}-\bar\lambda_h
\frac{b(\bar u_h-\psi_h,\bar u_h-\psi_h)}{b(\psi_h,\psi_h)}.
\end{eqnarray*}
\end{lemma}

\subsection{Newton iteration for eigenvalue problem}
This subsection introduces the main idea that deduces our numerical method in this paper.
Here, we use the Newton iteration idea to solve the eigenproblem \eqref{weak_problem}:

\noindent Find $(\lambda,u)\in\mathbb{R}\times V$ such that
\begin{equation}\label{Problem_Mixed_Form}
\left\{
\begin{array}{rcl}
a(u,v)-\lambda b(u,v)&=&0,\ \ \ \ \forall v\in V, \\
b(u,u)-1&=&0.
\end{array}
\right.
\end{equation}
If we have an eigenpair approximation $(\mu_0,u_0)$ with $b(u_0,u_0)=1$,
the Newton iteration method
for \eqref{Problem_Mixed_Form} is to find
$(\widetilde{\lambda}, \widetilde{u})\in \mathbb{R}\times V$
 such that
\begin{equation}\label{Newton_Iteration_Mixed}
\left\{
\begin{array}{rcl}
a(\widetilde{u}-u_0,v)-\mu_0\cdot b(\widetilde{u}-u_0,v)
-(\widetilde{\lambda}-\mu_0)b(u_0,v)
&=&\\
-(a(u_0,v)-\mu_0\cdot b(u_0,v)),&& \forall v\in V,\\
-b(\widetilde{u}-u_0,u_0)&=&0.
\end{array}
\right.
\end{equation}
After simplifying \eqref{Newton_Iteration_Mixed}, we have the following
equation for the new eigenpair
approximation $(\widetilde{\lambda},\widetilde{u})\in\mathbb{R}\times V$
\begin{equation}\label{Newton_Iteration_Mixed_2}
\left\{
\begin{array}{rcl}
a(\widetilde{u},v)-\mu_0\cdot b(\widetilde{u},v)-\widetilde{\lambda}
b(u_0,v)&=&-\mu_{0} b(u_0,v),\ \ \ \forall v\in V, \\
-b(\widetilde{u}-u_0,u_0)&=&0.
\end{array}
\right.
\end{equation}
Now, we come to prove that the mixed problem \eqref{Newton_Iteration_Mixed_2} has only one solution.
For this aim, we define the following bilinear forms
\begin{equation}\label{Mixed_Bilinear_Forms}
A_{\mu_{0}}(u,v)=a(u,v)-\mu_{0} b(u,v), \ \ \quad\ \  B(v,\nu)=-\nu b(u_0,v),
\end{equation}
where $u\in V$, $v\in V$, $\nu\in W=\mathbb{R}$ and
$\mu_{0}=\frac{a(u_0,u_0)}{b(u_0,u_0)}$.

Assume that $f\in V'$ and $g\in W'$. We consider the following mixed problem:

\noindent Find $(u,\lambda)\in V\times W$ such that
\begin{equation}\label{mixing_form}
\left\{
\begin{array}{rcl}
A_{\mu_{0}}(u,v)+B(v,\lambda)&=&f(v),\ \ \ \ \ \forall v\in V,\\
B(u,\nu)&=&g(\nu),\ \ \ \ \ \forall \nu\in W.
\end{array}
\right.
\end{equation}
About the existence and uniqueness of problem \eqref{mixing_form},
 the following theorem holds.
\begin{theorem}\label{LBB_Conditions_Theorem}
Assume $u_0$ is an eigenfunction approximation to $M(\lambda_1)$ with
sufficiently small error and $\|u_0\|_b=1$. Then the bilinear forms defined in
\eqref{Mixed_Bilinear_Forms} satisfy the following conditions
\begin{enumerate}
\item There exists $\alpha>0$ (depends on $\lambda_2-\lambda_1$) such that
\begin{equation}\label{A_sup}
A_{\mu_{0}}(v,v)\geq \alpha \|v\|_a^2,\ \ \ \ \ \forall v\in V_0,
\end{equation}
where $V_0=\{v | B(v,\nu)=0,\ \forall \nu\in W\}=\{v|b(u_0,v)=0\}$.
\item There exists $\sigma>0$ (depends on $1/\mu_0$) such that
\begin{equation}\label{B_sup}
\sup_{v\in V}\frac{B(v,\nu)}{\|v\|_a}\geq \sigma |\nu|,
\ \ \ \ \ \forall \nu\in W.
\end{equation}
\end{enumerate}
Based on these two conditions, the mixed equation \eqref{mixing_form} has only one
solution.
\end{theorem}
\begin{proof}

We decompose $u_0$ as $u_0=w_{1}+w_{1}^{\perp}$ such that $w_{1}\in M(\lambda_1)$ and
$w_{1}^{\perp}\perp M(\lambda_1)$.

Since $u_0$ ($\|u_0\|_b=1$) is an eigenfunction approximation to $M(\lambda_1)$ with sufficiently
small error, there exists 
a small enough number $\delta$ such that
\begin{eqnarray}\label{Error_u_0}
\|u_0-w_1\|_a&\leq& \delta.
\end{eqnarray}
From Lemma \ref{Rayleigh_Quotient_error_theorem}, we also have
\begin{eqnarray}\label{Error_mu}
|\mu_{0}-\lambda_1|&\leq& C\delta^2.
\end{eqnarray}
Since \eqref{Error_u_0} and
$\|u_0\|_b^2=\|w_{1}\|_b^2+\|w_{1}^{\perp}\|_b^2$,
$w_{1}^{\perp}$ and $w_{1}$
have estimates
$$\|w_1^{\perp}\|_b\leq C\|w_{1}^{\perp}\|_a\leq C\delta,\quad  \|w_{1}\|_b\geq 1-C\delta.$$
We also do the decomposition $v=v_1+v_1^{\perp}$ with
$v_1\in M(\lambda_1)$ and $v_1^{\perp}\perp M(\lambda_1)$ for $v\in V_0$.
Since $b(w_{1}+w_{1}^{\perp},v_1+v_1^{\perp})=0$, we have
\begin{eqnarray*}
\|v_1\|_b\|w_{1}\|_b=|b(v_1,w_{1})|=|-b(v_1^{\perp},w_{1}^{\perp})|
= |-b(v,w_{1}^{\perp})|
\leq C\delta\|v\|_b.
\end{eqnarray*}
Then $\|v_1\|_b$ has the following estimate
\begin{eqnarray}\label{Estimate_v_1}
\|v_1\|_b \leq \frac{C\delta}{1-C\delta}\|v\|_b \leq C\delta \|v\|_b.
\end{eqnarray}
From \eqref{Estimate_v_1} and the property $\|v\|_b^2=\|v_1\|_b^2+\|v_1^{\perp}\|_b^2$,
 the following estimates hold
\begin{eqnarray*}
b(v,v)&=&b(v_1,v_1)+b(v_1^{\perp},v_1^{\perp})\nonumber\\
&\leq& C \delta^2b(v,v)+\frac{1}{\lambda_2}a(v_1^{\perp},v_1^{\perp})\nonumber\\
&\leq& C \delta^2b(v,v)+\frac{1}{\lambda_2}a(v,v).
\end{eqnarray*}
Thus we have the following inequality
\begin{eqnarray}\label{Error_v_a_b}
b(v,v) &\leq& \frac{1}{\lambda_2(1-C\delta^2)}a(v,v).
\end{eqnarray}
From \eqref{Error_mu}, \eqref{Error_v_a_b} and the definition of $A_{\mu}(\cdot,\cdot)$,
the following inequalities hold
\begin{eqnarray*}\label{Estimate_A_mu}
a(v,v)-\mu_{0} b(v,v)&\geq&\Big(1-\frac{\mu_{0}}{\lambda_2(1-C\delta^2)}\Big)a(v,v)\nonumber\\
&\geq& \frac{\lambda_2(1-C\delta^2)-\mu_{0}}{\lambda_2(1-C\delta^2)}a(v,v)\nonumber\\
&\geq&\frac{\lambda_2-\lambda_1-C\delta^2}{\lambda_2(1-C\delta^2)}a(v,v).
\end{eqnarray*}
It means \eqref{A_sup} holds for
$\alpha=(\lambda_2-\lambda_1-C\delta^2)/(\lambda_2(1-C\delta^2))$
when $\delta$ is small enough.

Now, we come to prove \eqref{B_sup}.
From the definitions of $B(\cdot,\cdot)$ and $\mu$, we have
\begin{equation*}
\sup_{v\in V}\frac{B(v,\nu)}{\|v\|_a}\geq |\nu|\frac{b(u_0,u_0)}{\|u_0\|_a}
=\frac{|\nu|}{\mu_{0}},\ \ \ \ \forall \nu\in W. 
\end{equation*}
It means that \eqref{B_sup} holds for
\[
\sigma=\frac{1}{\mu_{0}}.
\]
From the theory for the mixed finite element method \cite{Brezzi2012Mixed}, there exists only one solution for
the equation \eqref{mixing_form}.
\end{proof}
\begin{corollary}\label{INF_SUP_Inequality_Corollary}
Under the conditions of Theorem \ref{LBB_Conditions_Theorem}, the following inequality holds
\begin{eqnarray}\label{INF_SUP_Inequality}
\|w\|_a+|\gamma| \leq C_4\sup_{0\neq (v,\nu)\in V\times W}\frac{A_{\mu_0}(w,v)+B(v,\gamma)+B(w,\nu)}{\|v\|_a+|\nu|},
\end{eqnarray}
for any $(w,\gamma)\in V\times W$. The constant $C_4$ depends on $1/(\lambda_2-\lambda_1)$ and $\lambda_1$.
\end{corollary}

\subsection{One Newton iteration step}
Based on the discussion in the last subsection, we propose an one correction step to improve the
given eigenpair approximation. Assume we have obtained an eigenpair approximation
$(\lambda_{1,h_k},u_{1,h_k})\in\mathbb{R}\times V_{h_k}$ with $\|u_{1,h_k}\|_b=1$.
Now we introduce a type of iteration step to improve the accuracy of the
current eigenpair approximation $(\lambda_{1,h_k},u_{1,h_k})$. Let
$V_{h_{k+1}}\subset V$ be a finer finite element space such that
$V_{h_k}\subset V_{h_{k+1}}$. Based on this finer finite element space,
we define the following one Newton iteration step.
\begin{algorithm}\label{Correction_Step}
One Newton Iteration Step

\begin{enumerate}
\item Solve the augmented mixed problem:\\
Find $(\widehat{\lambda}_{1,h_{k+1}},\widehat{u}_{1,h_{k+1}}) \in
\mathbb{R}\times V_{h_{k+1}}$ such that
\begin{eqnarray}\label{JD_problem}
\left\{
\begin{array}{lcl}
a(\widehat{u}_{1,h_{k+1}},v_{h_{k+1}})
-\lambda_{1,h_k}b(\widehat{u}_{1,h_{k+1}},v_{h_{k+1}})
-\widehat{\lambda}_{1,h_{k+1}}b(u_{1,h_k},v_{h_{k+1}})&&\\
\quad\quad\quad\quad \ \ \ \ \ \ \ \ \ =-\lambda_{1,h_k}b(u_{1,h_k},v_{h_{k+1}}),
\ \ \ \ \ \ \forall v_{h_{k+1}}\in V_{h_{k+1}},&&\\
b(\widehat{u}_{1,h_{k+1}},u_{1,h_k})=b(u_{1,h_k},u_{1,h_k}).&&
\end{array}
\right.
\end{eqnarray}
\item Do the normalization for $\widehat{u}_{1,h_{k+1}}$ as
\begin{eqnarray}\label{Normalization}
u_{1,h_{k+1}}=\frac{\widehat{u}_{1,h_{k+1}}}{\|\widehat{u}_{1,h_{k+1}}\|_b}
\end{eqnarray}
and compute the Rayleigh quotient for $u_{1,h_{k+1}}$
\begin{eqnarray}\label{Rayleigh_Quotient_u_h_k_1}
\lambda_{1,h_{k+1}}&=&\frac{a(u_{1,h_{k+1}},u_{1,h_{k+1}})}
{b(u_{1,h_{k+1}},u_{1,h_{k+1}})}.
\end{eqnarray}
\end{enumerate}
Then we obtain a new eigenpair approximation
$(\lambda_{1,h_{k+1}},u_{1,h_{k+1}})\in \mathbb{R}\times V_{h_{k+1}}$.
Summarize the above two steps into
\begin{eqnarray*}
(\lambda_{1,h_{k+1}},u_{1,h_{k+1}})={\it
Newton\_Iteration}(\lambda_{1,h_k},u_{1,h_k},V_{h_{k+1}}).
\end{eqnarray*}
\end{algorithm}
\begin{theorem}\label{Error_Estimate_One_Correction_Theorem}
Assume $(\lambda_{1,h_k},u_{1,h_k})$ is a good enough approximation to $(\lambda_1,u_1)$
such that \eqref{A_sup}, \eqref{B_sup} hold and $\lambda_{1,h_k}=a(u_{1,h_k},u_{1,h_k})/b(u_{1,h_k},u_{1,h_k})$.
After one iteration step, the resultant approximation
$(\lambda_{1,h_{k+1}},u_{1,h_{k+1}})\in\mathbb{R}\times V_{h_{k+1}}$
has the following error estimates
\begin{eqnarray}
\|\bar u_{1,h_{k+1}}-u_{1,h_{k+1}}\|_a &\leq &C_5 \|\bar u_{1,h_{k+1}}-u_{1,h_k}\|_a^2,\label{Error_k_k+1_1}\\
|\bar\lambda_{1,h_{k+1}}-\lambda_{1,h_{k+1}}|&\leq& C_6 \|\bar u_{1,h_{k+1}}-u_{1,h_k}\|_a^4,\label{Error_k_k+1_3}
\end{eqnarray}
where $C_5$ and $C_6$ are constants which depend on $1/(\lambda_2-\lambda_1)$ and $\lambda_1$ but are independent of the mesh sizes $h_k$
and $h_{k+1}$.
\end{theorem}
\begin{proof}
From the definition (\ref{weak_problem_Discrete}), we know the eigenpair approximation $(\bar\lambda_{1,h_{k+1}},\bar u_{1,h_{k+1}})$
satisfies the following equations
\begin{eqnarray}\label{JD_problem_Direct}
\left\{
\begin{array}{lcl}
a(\bar u_{1,h_{k+1}},v_{h_{k+1}})
-\lambda_{1,h_k}b(\bar u_{1,h_{k+1}},v_{h_{k+1}})
-\bar{\lambda}_{1,h_{k+1}}b(u_{1,h_k},v_{h_{k+1}})&&\\
\quad\quad\quad =(\bar{\lambda}_{1,h_{k+1}}-\lambda_{1,h_k})b(\bar u_{1,h_{k+1}},v_{h_{k+1}})
-\bar{\lambda}_{1,h_{k+1}}b(u_{1,h_k},v_{h_{k+1}}),&&\\
\quad\quad\quad\ \ \ \ \ \ \forall v_{h_{k+1}}\in V_{h_{k+1}},&&\\
b(\bar{u}_{1,h_{k+1}},u_{1,h_k})=b(\bar u_{1,h_{k+1}},u_{1,h_k}).&&
\end{array}
\right.
\end{eqnarray}
Let us define $w_{h_{k+1}}:=\bar u_{1,h_{k+1}}-\widehat{u}_{1,h_{k+1}}$ and
$\gamma:=\bar\lambda_{1,h_{k+1}}-\widehat{\lambda}_{1,h_{k+1}}$. From
(\ref{JD_problem}) and (\ref{JD_problem_Direct}), the following equations hold
\begin{eqnarray}\label{JD_problem_Difference}
\left\{
\begin{array}{lcl}
a(w_{h_{k+1}},v_{h_{k+1}})
-\lambda_{1,h_k}b(w_{h_{k+1}},v_{h_{k+1}})
-\gamma b(u_{1,h_k},v_{h_{k+1}})&&\\
\quad =(\bar{\lambda}_{1,h_{k+1}}-\lambda_{1,h_k})b(\bar u_{1,h_{k+1}}-u_{1,h_k},v_{h_{k+1}}),
\ \ \ \ \ \ \forall v_{h_{k+1}}\in V_{h_{k+1}},&&\\
\nu b(w_{h_{k+1}},u_{1,h_k})=\nu b(\bar u_{1,h_{k+1}}-u_{1,h_k},u_{1,h_k})&&\\
\quad\quad\quad\ \ \ \ \ =-\nu\frac{1}{2}b(\bar u_{1,h_{k+1}}-u_{1,h_k},\bar u_{1,h_{k+1}}-u_{1,h_k}),\ \ \ \ \ \forall \nu\in W.&&
\end{array}
\right.
\end{eqnarray}
Then combining Lemma \ref{Rayleigh_Quotient_error_theorem},
Corollary \ref{INF_SUP_Inequality_Corollary}, (\ref{JD_problem_Difference})
and
\begin{eqnarray*}
\|\bar u_{1,h_{k+1}}-u_{1,h_k}\|_b\lesssim \|\bar u_{1,h_{k+1}}-u_{1,h_k}\|_a,
\end{eqnarray*}
we have the following inequality
\begin{eqnarray}\label{Inequality_1}
\|w_{h_{k+1}}\|_a+|\gamma| &\lesssim& |\bar{\lambda}_{1,h_{k+1}}-\lambda_{1,h_k}|
\|\bar u_{1,h_{k+1}}-u_{1,h_k}\|_b
+\|\bar u_{1,h_{k+1}}-u_{1,h_k}\|_b^2\nonumber\\
&\lesssim& \|\bar u_{1,h_{k+1}}-u_{1,h_k}\|_a^2.
\end{eqnarray}
It means the following inequality holds
\begin{eqnarray}\label{Inequality_2}
\|\bar u_{1,h_{k+1}}-\widehat{u}_{1,h_{k+1}}\|_a
 &\lesssim & \|\bar u_{1,h_{k+1}}-u_{1,h_k}\|_a^2.
\end{eqnarray}
Combining the above inequality (\ref{Inequality_2}), the definition (\ref{Normalization}),
$\|\bar u_{1,h_{k+1}}\|_b=1$ and
$\|\widehat u_{1,h_{k+1}}\|_b\geq \|\bar u_{1,h_{k+1}}\|_b-\|\bar u_{1,h_{k+1}}-\widehat u_{1,h_k}\|_b$
 having a lower bound from zero,
we have the following inequalities
\begin{eqnarray*}
&&\|\bar u_{1,h_{k+1}}-u_{1,h_{k+1}}\|_a\nonumber\\
&\leq& \Big\|\bar u_{1,h_{k+1}}
-\frac{\bar u_{1,h_{k+1}}}{\|\widehat{u}_{1,h_{k+1}}\|_b}\Big\|_a
+\frac{\|\bar u_{1,h_{k+1}}-\widehat{u}_{1,h_{k+1}}\|_a}{\|\widehat{u}_{1,h_{k+1}}\|_b}\nonumber\\
&\leq& \frac{\|\bar u_{1,h_{k+1}}\|_a}{\|\widehat{u}_{1,h_{k+1}}\|_b}
\Big|\|\bar u_{1,h_{k+1}}\|_b-\|\widehat{u}_{1,h_{k+1}}\|_b\Big|
+\frac{\|\bar u_{1,h_{k+1}}-\widehat{u}_{1,h_{k+1}}\|_a}{\|\widehat{u}_{1,h_{k+1}}\|_b}\nonumber\\
&\leq& \frac{\|\bar u_{1,h_{k+1}}\|_a}{\|\widehat{u}_{1,h_{k+1}}\|_b}
\|\bar u_{1,h_{k+1}}-\widehat{u}_{1,h_{k+1}}\|_b
+\frac{\|\bar u_{1,h_{k+1}}-\widehat{u}_{1,h_{k+1}}\|_a}{\|\widehat{u}_{1,h_{k+1}}\|_b}\nonumber\\
&\lesssim& \|\bar u_{1,h_{k+1}}-\widehat{u}_{1,h_{k+1}}\|_a
\lesssim \|\bar u_{1,h_{k+1}}-u_{1,h_k}\|_a^2.
\end{eqnarray*}
This is the desired result (\ref{Error_k_k+1_1}). Furthermore, from  (\ref{Error_k_k+1_1})
and Lemma \ref{Rayleigh_Quotient_error_theorem}, the other desired result (\ref{Error_k_k+1_3})
can be obtained easily and the proof is complete.
\end{proof}
\begin{remark}
Theorem \ref{Error_Estimate_One_Correction_Theorem} shows that the Newton iteration method has
second order convergence rate when the initial approximation has enough accuracy. We also
would like to say that Theorem \ref{Error_Estimate_One_Correction_Theorem} and its proof also
give the analysis for the algebraic eigenvalue problems by the Newton iteration method.
\end{remark}
\section{Multilevel iteration method}

In this section, we introduce a type of multilevel scheme based on the
{\it One Newton Iteration Step} defined by Algorithm \ref{Correction_Step}.
The proposed multigrid method can obtain  eigenpair approximation with the
optimal accuracy and with much smaller computational work compared with
solving the eigenvalue problem directly in the finest finite element space.

Before introducing the multigrid scheme, we define a sequence of triangulations
$\mathcal{T}_{h_k}$ of $\Omega$. Suppose $\mathcal{T}_{h_1}$ is given and let
$\mathcal{T}_{h_k}$ be obtained from $\mathcal{T}_{h_{k-1}}$ via regular refinement
(produce $\beta^d$ subelements) such that
\[ h_k=\frac{1}{\beta}h_{k-1}.\]
Based on this sequence of meshes, we construct the corresponding nested linear
finite element spaces such that
\begin{eqnarray}\label{FEM_Space_Series}
V_{h_1}\subset V_{h_2}\subset\cdots\subset V_{h_n},
\end{eqnarray}
and the following relation of approximation errors hold
\begin{eqnarray}\label{Error_k_k_1}
\frac{1}{\beta}\eta_a(h_{k-1})\leq C_7\eta_a(h_k),\ \ \ \
\frac{1}{\beta}\delta_{h_{k-1}}(\lambda)\leq C_7\delta_{h_k}(\lambda),\ \ \ k=2,\cdots,n.
\end{eqnarray}
From the error estimate results in Proposition \ref{Error_estimate_Proposition}, we have
\begin{eqnarray}\label{Spectral_Projection_Estimate}
\|\bar u_{1,h_k}-\bar u_{1,h_{k+1}}\|_a&\leq& C_8\delta_{h_k}(\lambda_1), \ \
k=1,\cdots,n-1,
\end{eqnarray}
where the constant $C_8$ is a constant independent of the mesh size $h_k$.

 \begin{algorithm}\label{Multi_Correction}
Multilevel Eigenvalue Iteration Scheme
\begin{enumerate}
\item Construct a series of nested finite element
spaces $V_{h_1}, V_{h_2},\cdots,V_{h_n}$ such that
\eqref{FEM_Space_Series} and \eqref{Error_k_k_1} hold.
\item Solve the following eigenvalue problem:

Find $(\lambda_{1,h_1},u_{1,h_1})\in \mathbb{R}\times V_{h_1}$ such that
$b(u_{1,h_1},u_{1,h_1})=1$ and
\begin{eqnarray}\label{Initial_Eigen_Problem}
a(u_{1,h_1},v_{h_1})&=&\lambda_{1,h_1}b(u_{1,h_1},v_{h_1}),\ \ \ \ \forall v_{h_1}\in V_{h_1}.
\end{eqnarray}

\item  Do $k=1,\cdots,n-1$

Obtain a new eigenpair approximation
$(\lambda_{1,h_{k+1}},u_{1,h_{k+1}})\in \mathbb{R}\times V_{h_{k+1}}$
by a Newton iteration step
\begin{eqnarray}
(\lambda_{1,h_{k+1}},u_{1,h_{k+1}})=Newton\_Iteration(\lambda_{1,k},u_{1,h_k},V_{h_{k+1}}).
\end{eqnarray}
End do
\end{enumerate}
Finally, we obtain an eigenpair approximation
$(\lambda_{1,h_n},u_{1,h_n})\in \mathbb{R}\times V_{h_n}$.
\end{algorithm}
\begin{theorem}
Assume $h_1$ is small enough such that $(\lambda_{1,h_1},u_{1,h_1})$ satisfies
conditions \eqref{A_sup} and \eqref{B_sup}. After implementing
Algorithm \ref{Multi_Correction},
 the resultant eigenpair approximation $(\lambda_{1,h_n},u_{1,h_n})$ has the following
error estimates
\begin{eqnarray}
\|u_{1,h_n}-\bar u_{1,h_n}\|_a &\leq&\delta_{h_n}(\lambda_1),\label{Multi_Correction_Err_fun}\\
|\lambda_{1,h_n}-\bar{\lambda}_{1,h_n}|&\leq& C_9\delta_{h_n}^2(\lambda_1),\label{Multi_Correction_Err_eigen}
\end{eqnarray}
when the mesh size $h_1$ is small enough.

Besides, there exists an
eigenfunction $u_1$ of \eqref{weak_problem} corresponding to
$\lambda_1$ such that the following final convergence results hold
\begin{eqnarray}
\|u_1-u_{1,h_n}\|_a &\leq&2\delta_{h_n}(\lambda_1),\label{Multi_Correction_Err_fun_Final}\\
|\lambda_1-\lambda_{1,h_n}|&\leq&2C_{10}\delta_{h_n}^2(\lambda_1).\label{Multi_Correction_Err_eigen_Final}
\end{eqnarray}
\end{theorem}
\begin{proof}
Let us prove \eqref{Multi_Correction_Err_fun} by the method of induction.
First, it is obvious that \eqref{Multi_Correction_Err_fun} holds for $n=1$ according to \eqref{Initial_Eigen_Problem}.
Then we assume that \eqref{Multi_Correction_Err_fun} holds for $n=k$. It means we have the following estimate
\begin{eqnarray}\label{Assumption}
\|\bar u_{1,h_k}- u_{1,h_k}\|_a&\leq& \delta_{h_k}(\lambda_1).
\end{eqnarray}
Now let us consider the case of $n=k+1$.
Combining \eqref{Spectral_Projection_Estimate}, \eqref{Assumption}
and the triangle inequality leads to the following estimates
\begin{eqnarray}\label{Estimate_1}
\|\bar u_{1,h_{k+1}}-u_{1,h_{k+1}}\|_a &\leq& C_5 \|u_{1,h_k}-\bar u_{1,h_{k+1}}\|_a^2\nonumber\\
&\leq&2 C_5 \|u_{1,h_k}-\bar u_{1,h_k}\|_a^2 +2C_5\|\bar u_{1,h_k}-\bar u_{1,h_{k+1}}\|_a^2\nonumber\\
&\leq&2C_5\delta_{h_k}^2(\lambda_1)+2C_5C_8^2\delta_{h_k}^2(\lambda_1)\nonumber\\
&\leq& 2C_5\big(1+C_8^2\big)\delta_{h_k}^2(\lambda_1)\nonumber\\
&=&\Big(2\beta C_5 \big(1+C_8^2\big)\delta_{h_k}(\lambda_1)\Big) \frac{\delta_{h_k}(\lambda_1)}{\beta}\nonumber\\
&\leq& \Big(2\beta C_5C_7\big(1+C_8^2\big)\delta_{h_k}(\lambda_1)\Big)\delta_{h_{k+1}}(\lambda_1).
\end{eqnarray}
This means that the result \eqref{Multi_Correction_Err_fun} also holds for $n=k+1$ if
$2\beta C_5C_7 \big(1+C_8^2\big)\delta_{h_k}(\lambda_1)<1$.
Thus we prove the desired result \eqref{Multi_Correction_Err_fun}.
From Lemma \ref{Rayleigh_Quotient_error_theorem}
and \eqref{Multi_Correction_Err_fun}, we can obtain the desired result
\eqref{Multi_Correction_Err_eigen}. Finally, \eqref{Multi_Correction_Err_fun_Final} and
\eqref{Multi_Correction_Err_eigen_Final} can be proved from \eqref{Eigenfunction_Error}, \eqref{Eigenvalue_Error},
\eqref{Multi_Correction_Err_fun}, \eqref{Multi_Correction_Err_eigen}
and the triangle inequality.
\end{proof}
\subsection{Multi eigenvalues}
Now, we turn to extend the Newton iteration \eqref{Newton_Iteration_Mixed} for solving one eigenvalue to the corresponding version for multi eigenvalues
(include simple and multiple eigenvalues).  Assume that $\lambda_{m}<\lambda_{m+1}$ and
we have obtained the first $m$ eigenpairs approximation $\{(\mu_j,u_{0,j})\}_{j=1}^{m}$ to the problem \eqref{Problem_Mixed_Form}, which satisfy
$$b(u_{0,i},u_{0,j})=\delta_{ij},\quad i,j=1,\cdots,m,$$
where $\mu_j$ is the Rayleigh quotient of $u_{0,j}$.

The Newton iteration method for \eqref{Problem_Mixed_Form} is to find
$(x_j,\widetilde{u}_j)\in \mathbb{R}^{m}\times V$ $(j=1,\cdots, m)$ such that

\begin{eqnarray}\label{Multi_Newton_Iteration_Mixed}
\hskip-0.5cm\left\{
\begin{array}{rcl}
a(\widetilde{u}_j,v)-\mu_j\cdot b(\widetilde{u}_j,v)-\sum_{i=1}^{m}x_{ij}
b(u_{0,i},v)&=&-\mu_j b(u_{0,j},v),\ \ \ \forall v\in V, \\
b(\widetilde{u}_{j},u_{0,i})&=&b({u}_{0,j},u_{0,i}), \ \ \forall  i=1,\cdots,m,
\end{array}
\right.
\end{eqnarray}
where $x_{ij}$ is the $i$-th component of $x_j$.

Now, we come to prove \eqref{Multi_Newton_Iteration_Mixed} has only one solution for any $j=1,\cdots,m$.
For this aim, we define the following bilinear forms
\begin{equation}\label{Multi_Mixed_Bilinear_Forms}
A_{\mu_{j}}(u,v)=a(u,v)-\mu_{j} b(u,v), \ \   B(v,y)=-\sum_{i=1}^{m}y_{i}b(u_{0,i},v),
\end{equation}
where $u\in V$, $v\in V$, $y\in W=\mathbb{R}^{m}$.

Assume that $f_{\mu_j}\in V^{'}$, $g_{j}\in W^{'}$ are defined as
\begin{equation*}
f_{\mu_j}(v)=-\mu_{j}b(u_{0,j},v),\ \ \ \ \ g_{j}(y)=-\sum_{i=1}^{m}y_ib(u_{0,i},u_{0,j}).
\end{equation*}

 We consider the following multi mixed problems:
 Find $(x_j,\widetilde{u}_j)\in \mathbb{R}^{m}\times V$, $(j=1,\cdots, m)$, such that
\begin{equation}\label{Multi_mixing_form}
\left\{
\begin{array}{rcl}
A_{\mu_{j}}(\widetilde{u}_j,v)+B(v,x)&=&f_{\mu_j}(v),\ \ \ \ \ \forall v\in V,\\
B(\widetilde{u}_j,y)&=&g_j(y),\ \ \ \ \ \ \forall y\in W.
\end{array}
\right.
\end{equation}

Define $\mathcal{K}=M(\lambda_1)\cup\cdots \cup M(\lambda_m)$. About the existence and
uniqueness of problem \eqref{Multi_mixing_form}, the following theorem holds.
\begin{theorem}
Assume that there exists a decomposition of eigenspace $\mathcal{K}$ satisfying
$\mathcal{K}=M(\lambda_1)\oplus\cdots\oplus M(\lambda_m)$ such that $u_{0,j}$ is
 an eigenfunction approximation to $M(\lambda_j)$ $(j=1,\cdots,m)$.
 Then the bilinear forms defined in \eqref{Multi_Mixed_Bilinear_Forms} satisfy the following conditions
\begin{enumerate}
\item There exists $\alpha>0$ such that
\begin{equation}\label{Multi_A_sup}
A_{\mu_{j}}(v,v)\geq \alpha \|v\|_a^2,\ \ \ \ \ \forall v\in V_0,
\end{equation}
where $V_0=\{v | B(v,y)=0,\ \forall y\in W\}=\{v|b(u_{0,i},v)=0,\ \forall i=1,\cdots\,m\}$.
\item There exists $\sigma >0$ such that
\begin{equation}\label{Multi_B_sup}
\sup_{v\in V}\frac{B(v,y)}{\|v\|_a}\geq \sigma \|y\|,
\ \ \ \ \ \forall y\in W,
\end{equation}
\end{enumerate}
where $\|y\|:=\max_{i\in\{1,\cdots, m\}}|y_i|.$

Based on these two conditions, for any $j$ $(j=1,\cdots,m)$, the multi mixed equations
 \eqref{Multi_mixing_form} have only one solution.
\end{theorem}
\begin{proof}

We decompose $u_{0,j}$ as $u_{0,j}=w_{0,j}+w_{0,j}^{\perp}$ such that $w_{0,j}\in M(\lambda_j)$
and $w_{0,j}^{\perp}\perp_b w_{0,j}$.
Then $\mathrm{span}\{w_{0,1},\cdots, w_{0,m}\}$ is an orthonormal basis of eigenspace $\mathcal{K}$.

Since $u_{0,j}$ ($\|u_{0,j}\|_b=1$) is an eigenfunction approximation to $M(\lambda_j)$ with sufficiently small error,
there is a small enough number $\delta$ such that
\begin{eqnarray}\label{Multi_Error_u_0}
\|u_{0,j}-w_{0,j}\|_a\leq \delta, \quad u_{0,j}-w_{0,j}\bot_b \ \mathrm{span}\{w_{0,j}\},\quad j=1,\cdots,m.
\end{eqnarray}

From Lemma \ref{Rayleigh_Quotient_error_theorem}, we also have
\begin{eqnarray}\label{Multi_Error_mu}
|\mu_{j}-\lambda_j|\leq C\delta^2,\quad j=1,\cdots,m .
\end{eqnarray}

Since \eqref{Multi_Error_u_0} and
$\|u_{0,j}\|_b^2=\|w_{0,j}\|_b^2+\|w_{0,j}^{\perp}\|_b^2$,
$w_{0,j}^{\perp}$ and $w_{0,j}$ have estimates
$$\|w_{0,j}^{\perp}\|_b\leq C\|w_{0,j}^{\perp}\|_a\leq C\delta, \quad \|w_{0,j}\|_b\geq 1-C\delta, \quad j=1,\cdots,m.$$
Similarly, we also do decomposition $v\in V_0$ as
$$v=v_1+\cdots+v_m+v^{*}=v_j+v_j^{\perp},\ \ j=1,\cdots, m$$
satisfying
$$v^{*}\perp_b \mathcal{K}, \quad v_j\in \mathrm{span}\{w_{0,j}\},  \quad v_j^{\perp}\perp_b \mathrm{span}\{w_{0,j}\}.$$
According to the definition of $v\in V_0$, i.e., $b(w_{0,j}+w_{0,j}^{\perp},v_j+v_j^{\perp})=0$, we have
\begin{eqnarray*}
\|v_j\|_b\|w_{0,j}\|_b&=&|b(v_j,w_{0,j})|=|-b(v_j^{\perp},w_{0,j}^{\perp})|
= |b(v,w_{0,j}^{\perp})|\nonumber\\
&\leq& C\delta\|v\|_b, \ \ j=1,\cdots,m.
\end{eqnarray*}
Furthermore,
\begin{eqnarray}\label{Multi_Estimate_v_1}
\|v_j\|_b \leq \frac{C\delta}{1-C\delta}\|v\|_b \leq C\delta \|v\|_b,\ \ \ j=1,\cdots,m.
\end{eqnarray}
From \eqref{Multi_Estimate_v_1} and the property $\|v\|_b^2=\|v_1\|_b^2+\cdots+\|v_m\|_b^2+\|v^{*}\|_b^2$,
the following estimates hold
\begin{eqnarray*}
b(v,v)&=&b(v_1,v_1)+\cdots+b(v_m,v_m)+b(v^{*},v^{*}) \nonumber\\
&\leq& m C \delta^2b(v,v)+\frac{1}{\lambda_{m+1}}a(v^{*},v^{*})\nonumber\\
&\leq& m C \delta^2b(v,v)+\frac{1}{\lambda_{m+1}}a(v,v).
\end{eqnarray*}
Thus we have the following inequality
\begin{eqnarray}\label{Multi_Error_v_a_b}
b(v,v) &\leq& \frac{1}{\lambda_{m+1}(1-m C\delta^2)}a(v,v).
\end{eqnarray}
From \eqref{Multi_Error_mu}, \eqref{Multi_Error_v_a_b} and the definition of $A_{\mu_{j}}(\cdot,\cdot)$, the following inequalities hold
\begin{eqnarray*}\label{Multi_Estimate_A_mu}
a(v,v)-\mu_{j} b(v,v)&\geq&\Big(1-\frac{\mu_{j}}{\lambda_{m+1}(1-m C\delta^2)}\Big)a(v,v)\nonumber\\
&\geq& \frac{\lambda_{m+1}(1-m C\delta^2)-\mu_{j}}{\lambda_{m+1}(1-m C\delta^2)}a(v,v)\nonumber\\
&\geq&\frac{\lambda_{m+1}-\lambda_j-C\delta^2}{\lambda_{m+1}(1-mC\delta^2)}a(v,v).
\end{eqnarray*}
It means \eqref{Multi_A_sup} holds for
$\alpha=(\lambda_{m+1}-\lambda_j-C\delta^2)/\big(\lambda_{m+1}(1-mC\delta^2)\big)>0$
$(j=1,\cdots,m)$, when $\delta$ is small enough.

Now, we come to prove \eqref{Multi_B_sup}. Assume that the index $s$ satisfies $\|y\|=|y_s|$.
 From $b(u_{0,i},u_{0,j})=\delta_{ij}$ $(i,j=1,\cdots,m)$ and the definition of $B(\cdot,\cdot)$ and $\mu_{j}$,
 taking $v=-\mathrm{sign}(y_s)u_{0,s}$, we have
\begin{equation*}
\sup_{v\in V}\frac{B(v,y)}{\|v\|_a}\geq \frac{|y_s|b(u_{0,s},u_{0,s})}{\|u_{0,s}\|_a}
=\frac{\|y\|}{\mu_{s}}\geq\frac{\|y\|}{\mu}>0,\ \ \ \ \forall y\in W,
\end{equation*}
where $\mu=\max_{t \in\{1,2,\cdots,m\}}\{\mu_t\}$. It means that \eqref{Multi_B_sup} holds for
\[
\sigma=\frac{1}{\mu}.
\]

From the theory for the mixed finite element method \cite{Brezzi2012Mixed}, there exists only one solution
 for the equations \eqref{Multi_mixing_form} for any $j=1,\cdots,m$.
\end{proof}

\subsection{Multilevel iteration for multi eigenvalues}
Based on the discussion in the last subsection, we extend the one iteration step to improve
given eigenpairs approximation to the first $m$ given eigenpair approximations. Assume we have obtained
$m$ eigenpairs approximation
$(\lambda_{i,h_k},u_{i,h_k})\in\mathbb{R}\times V_{h_k}$ with $\|u_{i,h_k}\|_b=1$ $(i=1,\cdots,m)$.
Now we introduce a type of iteration step to improve the accuracy of the
current eigenpair approximation $(\lambda_{i,h_k},u_{i,h_k})$. Let
$V_{h_{k+1}}\subset V$ be a finer finite element space such that
$V_{h_k}\subset V_{h_{k+1}}$. Based on this finer finite element space,
we define the following one Newton iteration step for multi eigenvalues.  We can state the following version
of {\it Multilevel Eigenvalue Iteration Scheme} for $m$ eigenvalues.

Similarly, we first give a type of
{\it One Iteration Step for Multi Eigenvalues} for the given
eigenpair approximations $\{\lambda_{i,h_k},u_{i,h_k}\}_{i=1}^{m}$.

\begin{algorithm}\label{Correction_Step_Multiple}
One Newton Iteration Step for Multi Eigenvalues
\begin{enumerate}
\item Do $i=1,\cdots,m$\\
Find $(x_{i,h_{k+1}},\widetilde{u}_{i,h_{k+1}})
\in \mathbb{R}\times V_{h_{k+1}}$ such that
\begin{eqnarray}\label{aux_problem_Multiple}
\left\{
\begin{array}{lcl}
a(\widetilde{u}_{i,h_{k+1}},v_{h_{k+1}})-\lambda_{i,h_{k}}
b(\widetilde{u}_{i,h_{k+1}},v_{h_{k+1}})-\sum_{s=1}^{m}x_{si,h_{k+1}}b(u_{s,h_k},v_{h_{k+1}})&&\\
\quad\quad\quad \ \ \ \ \ \ \ \ =-\lambda_{i,h_{k}}b(u_{i,h_k},v_{h_{k+1}}),
\ \ \ \ \ \ \forall v_{h_{k+1}}\in V_{h_{k+1}},&&\\
b(\widetilde{u}_{i,h_{k+1}},u_{j,h_k})=\delta_{ij},\quad\quad\quad\quad\ \   \forall j=1,\cdots, m,&&
\end{array}
\right.
\end{eqnarray}
where $x_{si,h_{k+1}}$ is the $s$-th component of $x_{i,h_{k+1}}$.\\
End Do

\item Build a finite dimensional space
$\widetilde{V}_{h_{k+1}}=\mathrm{span}\{\widetilde{u}_{1,h_{k+1}}, \cdots, \widetilde{u}_{m,h_{k+1}}\}$
and solve the following eigenvalue problem:\\
Find $(\lambda_{i,h_{k+1}},u_{i,h_{k+1}})\in \mathbb{R}\times \widetilde{V}_{h_{k+1}}$,
$i=1,2,\cdots,m$, such that
$b(u_{i,h_{k+1}},u_{i,h_{k+1}})=1$ and
\begin{eqnarray*}\label{Initial_Eigen_Problem_Multiple}
a(u_{i,h_{k+1}},v_{h_{k+1}})&=&\lambda_{i,h_{k+1}}b(u_{i,h_{k+1}},v_{h_{k+1}}),
\  \ \forall v_{h_{k+1}}\in \widetilde{V}_{h_{k+1}}.
\end{eqnarray*}

\end{enumerate}
We summarize above two steps into
\begin{eqnarray*}
\{\lambda_{i,h_{k+1}},u_{i,h_{k+1}}\}_{i=1}^{m}={\it
Newton\_Iteration}(\{\lambda_{i,h_k},u_{j,h_k}\}_{i=1}^{m},V_{h_{k+1}}).
\end{eqnarray*}
\end{algorithm}

Based on Algorithm \ref{Correction_Step_Multiple}, we come to give the corresponding
multilevel correction method.
\begin{algorithm}\label{Multi_Correction_Multiple}
Multilevel Eigenvalue Iteration Scheme for Multi Eigenvalues

\begin{enumerate}
\item Construct a series of nested finite element
spaces $V_{h_1}, V_{h_2},\cdots,V_{h_n}$ such that
\eqref{FEM_Space_Series} and \eqref{Error_k_k_1} hold.
\item Solve the  eigenvalue problem in the initial finite element space $V_{h_1}$:\\
Find $(\lambda_{h_1},u_{h_1})\in \mathbb{R}\times V_{h_1}$ such that
$b(u_{h_1},u_{h_1})=1$ and
\begin{eqnarray*}
a(u_{i,h_1},v_{h_1})&=&\lambda_{i,h_1}b(u_{i,h_1},v_{h_1}),\ \ \ \ \forall v_{h_1}\in V_{h_1}.
\end{eqnarray*}
Choose the first $m$ eigenpairs $\{\lambda_{i,h_1},u_{i,h_1}\}_{i=1}^{m}$ which approximate
the desired eigenpairs.

\item  Do $k=1,\cdots,n-1$

Obtain new eigenpair approximations
$\{\lambda_{i,h_{k+1}},u_{i,h_{k+1}}\}_{i=1}^{m}\in \mathbb{R}\times V_{h_{k+1}}$
by the one Newton iteration step defined in Algorithm \ref{Correction_Step_Multiple}
\begin{eqnarray*}
\{\lambda_{i,h_{k+1}},u_{i,h_{k+1}}\}_{i=1}^{m}={\it
Newton\_Iteration}(\{\lambda_{i,h_k},u_{i,h_k}\}_{i=1}^{m},V_{h_{k+1}}).
\end{eqnarray*}
End do
\end{enumerate}
Finally, we obtain $m$ eigenpair approximations
$\{\lambda_{i,h_n},u_{i,h_n}\}_{i=1}^{m}\in \mathbb{R}\times V_{h_n}$.
\end{algorithm}

In  Algorithm \ref{Correction_Step_Multiple}, the parallel computation can
 be used to solve \eqref{aux_problem_Multiple} for different $i$. The analysis of the scheme
 for multi eigenvalues will be given in our future work.

\section{Work estimate of multilevel eigenvalue iteration scheme}
In this section, we turn our attention to the estimate of computational work
for Algorithm \ref{Multi_Correction} $($Algorithm \ref{Multi_Correction_Multiple}$)$.
 We will show that
Algorithm \ref{Multi_Correction} $($Algorithm \ref{Multi_Correction_Multiple}$)$
makes solving the eigenvalue problem need almost the
optimal computational work if solving the linear equation \eqref{JD_problem}
needs only the linear computational work.

First, we investigate the dimension of each level linear
finite element space as $N_k:=\mathrm{dim}V_{h_k}$. Then the following property holds
\begin{eqnarray}\label{relation_dimension}
N_k \thickapprox\Big(\frac{1}{\beta}\Big)^{d(n-k)}N_n,\ \ \ k=1,2,\cdots, n.
\end{eqnarray}
\begin{theorem}
Assume solving the eigenvalue problem in the coarse space $V_{h_1}$ needs work
$\mathcal{O}(M_{h_1})$ and the work for solving the linear equation
 \eqref{JD_problem} $($when $m>1$, for each $i$, using parallel technique to solve
 \eqref{aux_problem_Multiple}$)$ in each level space $V_{h_k}$ is only
$\mathcal{O}(N_k)$ for $k=2,\cdots,n$. Then the work involved in
Algorithm \ref{Multi_Correction} $($Algorithm \ref{Multi_Correction_Multiple}$)$
 is $\mathcal{O}(N_n+M_{h_1})$. Furthermore, the complexity
will be $\mathcal{O}(N_n)$ provided $M_{h_1}\leq N_n$.
\end{theorem}

\begin{proof}
Let $W_k$ denote the work of the iteration step defined in Algorithm
\ref{Correction_Step} $($Algorithm \ref{Correction_Step_Multiple} in each computing node$)$
in the $k$-th finite element space $V_{h_k}$ for $k=2, \cdots, n$.
From the iteration definition in Algorithm \ref{Correction_Step}
$($Algorithm \ref{Correction_Step_Multiple}$)$, we have
\begin{eqnarray}\label{work_k}
W_k&=&\mathcal{O}(N_k),\ \ \ \ \mbox{ for}\ k=2, \cdots, n.
\end{eqnarray}
Iterating \eqref{work_k} and using the fact \eqref{relation_dimension},
the following estimates hold
\begin{eqnarray}\label{Work_Estimate}
\mbox{Total work} &=&\sum_{k=1}^nW_k= \mathcal{O}\Big(M_{h_1}+\sum_{k=2}^nN_k\Big)
=\mathcal{O}\Big(M_{h_1}+\sum_{k=2}^nN_k\Big)\nonumber\\
&=&\mathcal{O}\Big(M_{h_1}+\sum_{k=2}^n\Big(\frac{1}{\beta}\Big)^{d(n-k)}N_n\Big)
=\mathcal{O}(N_n+M_{h_1}).
\end{eqnarray}
This is the desired estimate $\mathcal{O}(N_n+M_{h_1})$ for the computational work and the
one $\mathcal{O}(N_n)$ can be derived with the condition $M_{h_1}\leq N_n$.
\end{proof}
\section{Numerical results}
In this section, two numerical examples are presented to illustrate the
efficiency of the multilevel iteration scheme proposed in this
paper.

\subsection{Model eigenvalue problem}
Here we give the numerical results of the  multilevel iteration
scheme for the Laplace eigenvalue problem
on the two dimensional domain $\Omega=(0,1 )\times (0, 1)$.  The sequence of
finite element spaces is constructed by
using linear element on the series of meshes which are produced by the
regular refinement with $\beta =2$ (producing $\beta^2$ subelements).
In this example, we use two meshes which are generated by Delaunay method as
the initial mesh $\mathcal{T}_{h_1}$ ($H=h_{1}$) to produce two sequences of finite
 element spaces for investigating the convergence behaviors.
Figure \ref{Initial_Mesh} shows the corresponding
initial meshes: one is coarse and the other is fine.

Algorithm \ref{Multi_Correction} is applied to solve the eigenvalue problem.
For comparison, we also solve the eigenvalue problem by the direct method.
\begin{figure}[htb]
\centering
\includegraphics[width=6cm,height=5cm]{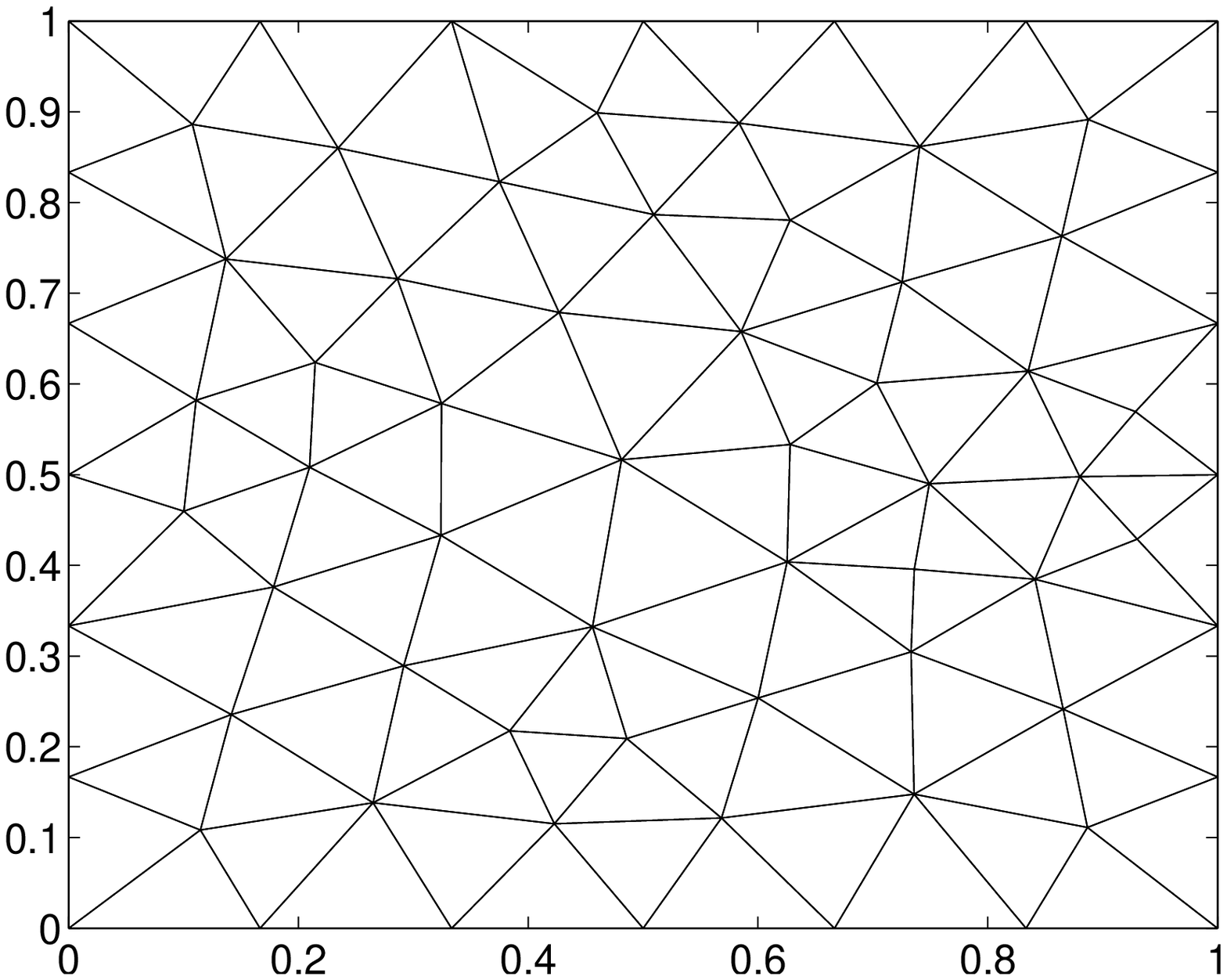}
\includegraphics[width=6cm,height=5cm]{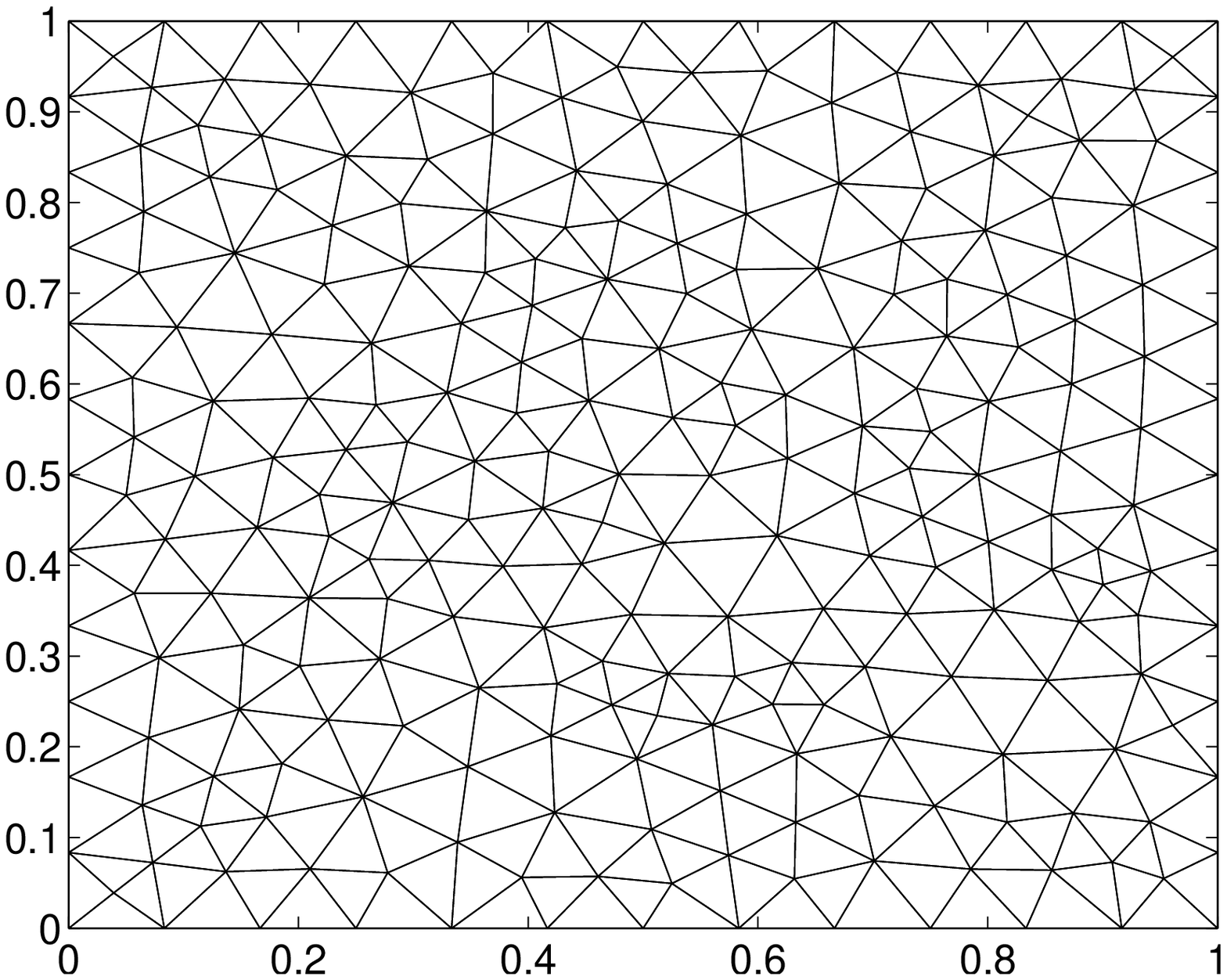}
\caption{\small\texttt The initial coarse $H=1/6$ and fine $H=1/12$ meshes for Example 1}
\label{Initial_Mesh}
\end{figure}

Figure \ref{numerical_multi_grid_2D}
gives the corresponding numerical results for the first eigenvalue
$\lambda_1=2\pi^2$ and the corresponding eigenfunction on the two initial meshes
 illustrated in Figure \ref{Initial_Mesh}. From Figure \ref{numerical_multi_grid_2D},
we find the  multilevel iteration scheme can obtain
the optimal error estimates as same as the direct eigenvalue solving method
for the eigenvalue and the corresponding eigenfunction approximations.

\begin{figure}[htb]
\centering
\includegraphics[width=6.5cm,height=6cm]{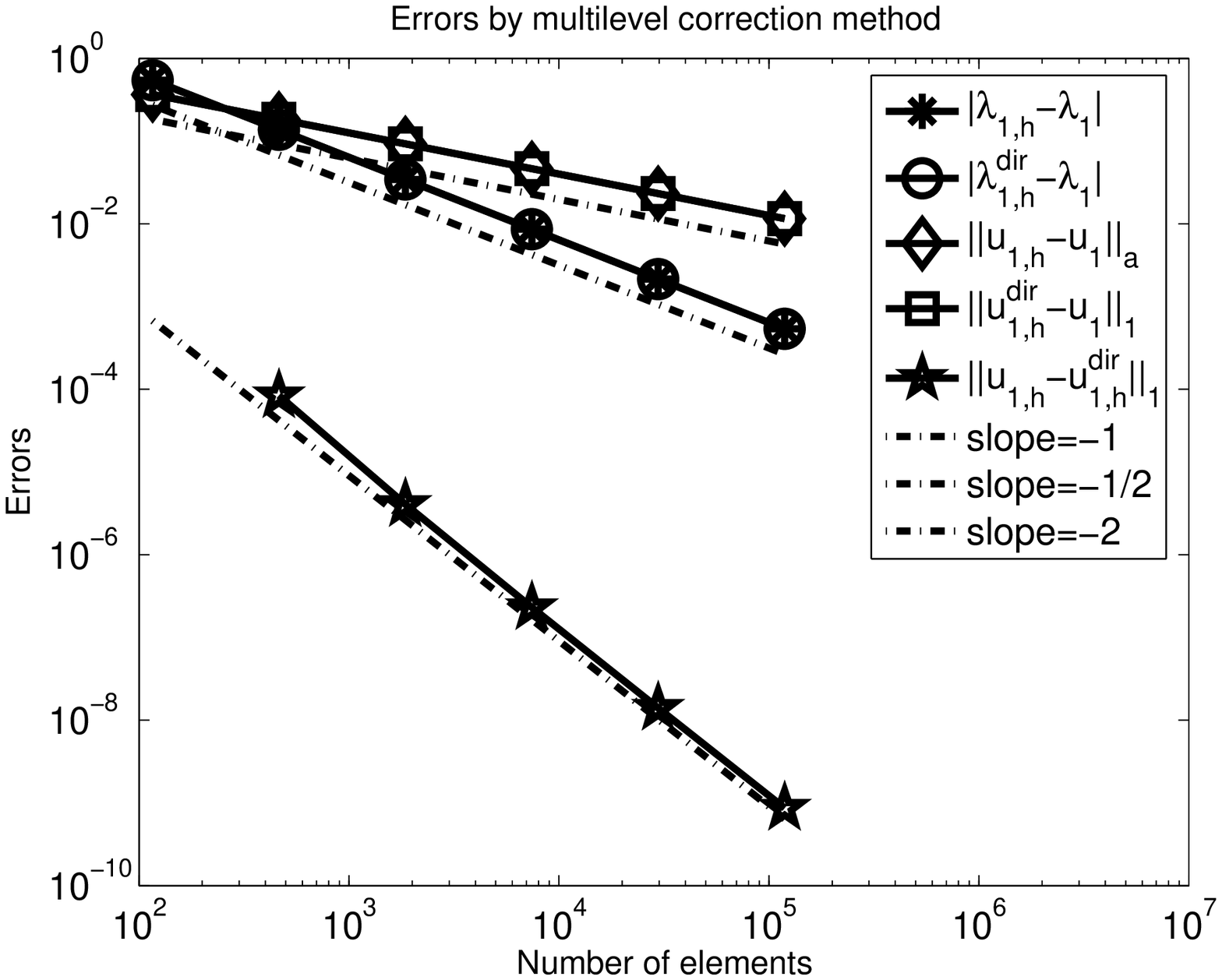}
\includegraphics[width=6.5cm,height=6cm]{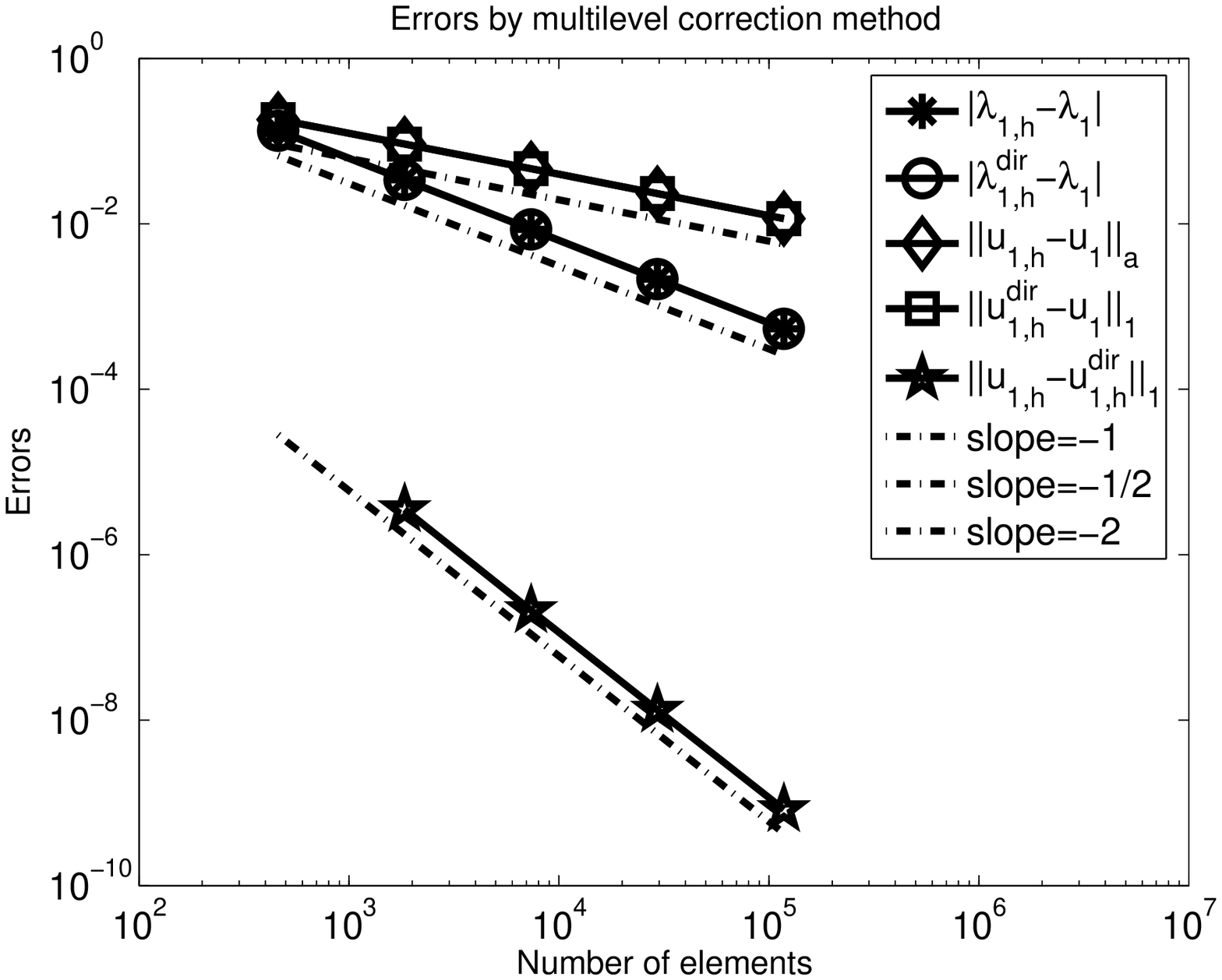}
\caption{\small\texttt The errors of the  multilevel iteration
algorithm for the first eigenvalue $2\pi^2$ and the corresponding eigenfunction,
where $u_h^{\mbox{dir}}$ and $\lambda_h^{\mbox{dir}}$ denote the eigenfunction
 and eigenvalue approximation by direct eigenvalue solving
 (The left subfigure is for the coarse initial mesh in the left of Figure \ref{Initial_Mesh}
and the right one for the fine initial mesh in the right of Figure \ref{Initial_Mesh})}
\label{numerical_multi_grid_2D}
\end{figure}

We also check the convergence behavior for multi eigenvalue approximations with Algorithm
\ref{Multi_Correction}. Here the first six eigenvalues
$\lambda=2\pi^2, 5\pi^2, 5\pi^2, 8\pi^2,10\pi^2,
10\pi^2$ are investigated. We adopt the meshes in Figure \ref{Initial_Mesh} as
the initial ones and the corresponding numerical results are shown
in Figure \ref{numerical_multi_grid_2D_6}.
Figure \ref{numerical_multi_grid_2D_6} also exhibits the
optimal convergence rate of the  multilevel iteration scheme.

\begin{figure}[htb]
\centering
\includegraphics[width=6.5cm,height=6cm]{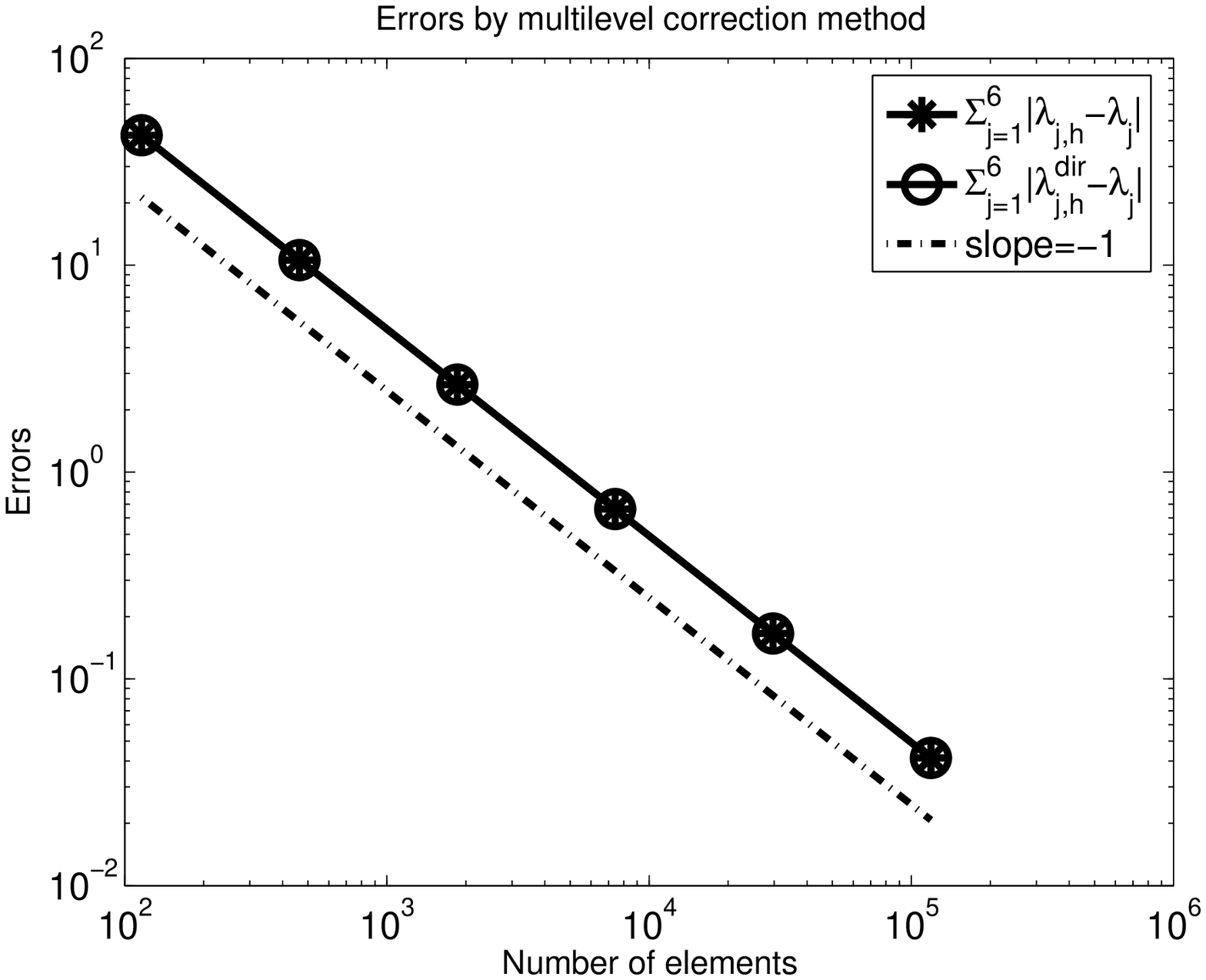}
\includegraphics[width=6.5cm,height=6cm]{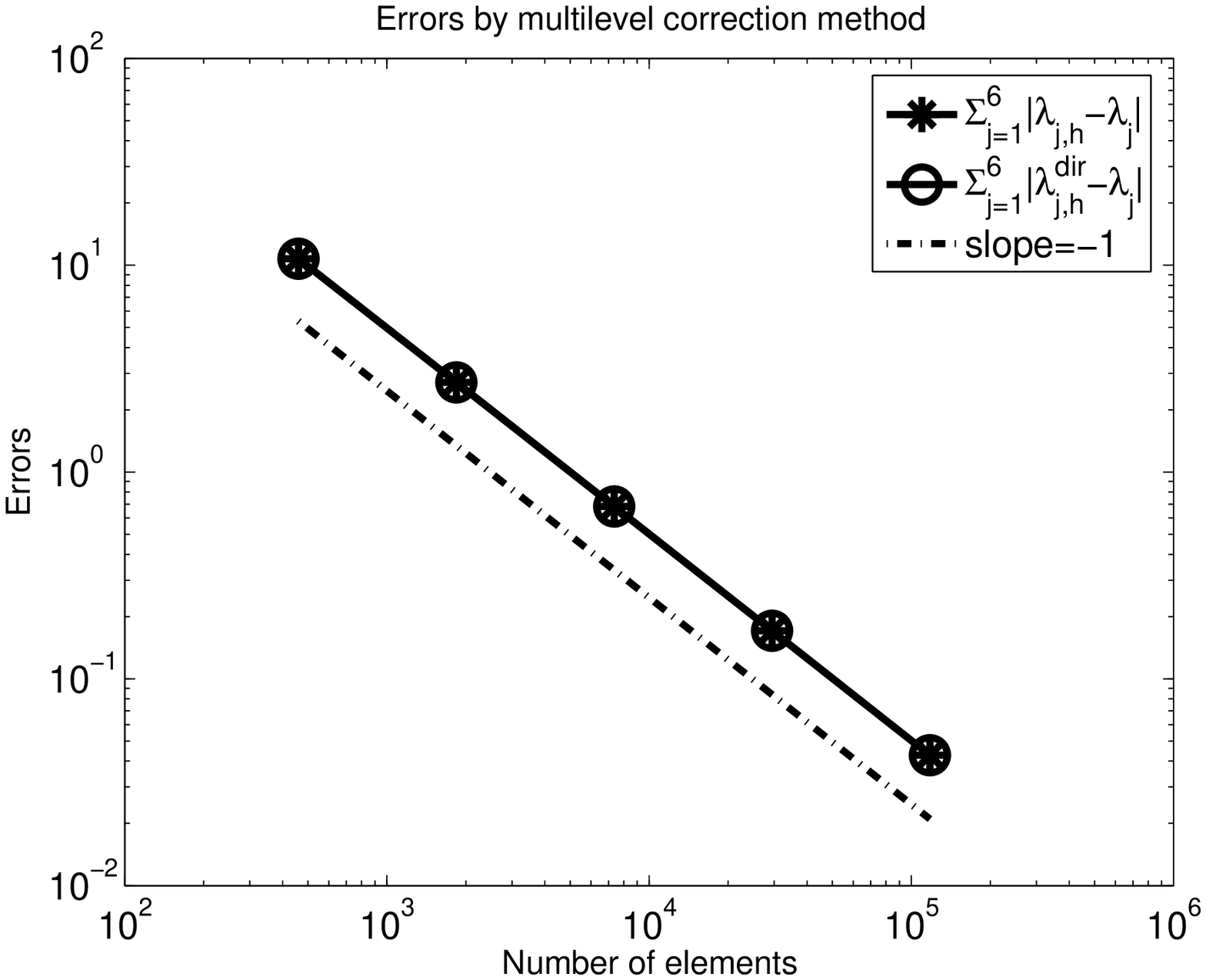}
\caption{\small\texttt The errors of the  multilevel iteration
algorithm for the first six eigenvalues on the unit square
(The left subfigure is for the coarse initial mesh in the left of Figure \ref{Initial_Mesh}
and the right one for the fine initial mesh in the right of Figure
\ref{Initial_Mesh})}\label{numerical_multi_grid_2D_6}
\end{figure}
\subsection{More general eigenvalue problem}
Here we give the numerical results of the  multilevel iteration
scheme for solving a more general eigenvalue problem on the unit square
domain $\Omega=(0, 1)\times (0, 1)$.

Find $(\lambda,u)$ such that
\begin{equation}\label{Example_2}
\left\{
\begin{array}{rcl}
-\nabla\cdot\mathcal{A}\nabla u+\phi u&=&\lambda\rho u,\quad\mbox{ in}\ \Omega,\\
u&=&0,\quad\ \ \ \mbox{ on}\ \partial\Omega,\\
\int_{\Omega}\rho u^2d\Omega&=&1,
\end{array}
\right.
\end{equation}
where
\begin{equation*}
\mathcal{A}=\left (
\begin{array}{cc}
$$1+(x_1-\frac{1}{2})^2$$&$$(x_1-\frac{1}{2})(x_2-\frac{1}{2})$$\\
$$(x_1-\frac{1}{2})(x_2-\frac{1}{2})$$&$$1+(x_2-\frac{1}{2})^2$$
\end{array}
\right),
\end{equation*}
$\phi=e^{(x_1-\frac{1}{2})(x_2-\frac{1}{2})}$ and
$\rho=1+(x_1-\frac{1}{2})(x_2-\frac{1}{2})$.

We first solve the eigenvalue problem
\eqref{Example_2} in the linear finite element space on the coarse mesh
$\mathcal{T}_{h_1}$. Then refine the mesh by the regular way to produce
a series of meshes $\mathcal{T}_{h_k}\ (k=2,\cdots,n)$ with $\beta =2$
(connecting the midpoints of each edge) and solve the augmented mixed
problem \eqref{JD_problem} in the finer linear finite element space
$V_{h_k}$ defined on $\mathcal{T}_{h_k}$.

In this example, we also use two coarse meshes which are shown in Figure \ref{Initial_Mesh}
as the initial meshes to investigate the convergence behaviors.
Since the exact solution is unknown, we choose an adequately accurate eigenvalue
approximations with the extrapolation method (see, e.g., \cite{Lin2006Finite}) as the exact eigenvalue.
Figure \ref{numerical_multi_grid_Exam_2} gives the corresponding
numerical results for the first six eigenvalue approximations and
their corresponding eigenfunction approximations.
Here we also compare the numerical results with the direct algorithm.
Figure \ref{numerical_multi_grid_Exam_2} also exhibits the optimal convergence rate of
Algorithm \ref{Multi_Correction}.
\begin{figure}[htb]
\centering
\includegraphics[width=6.5cm,height=6cm]{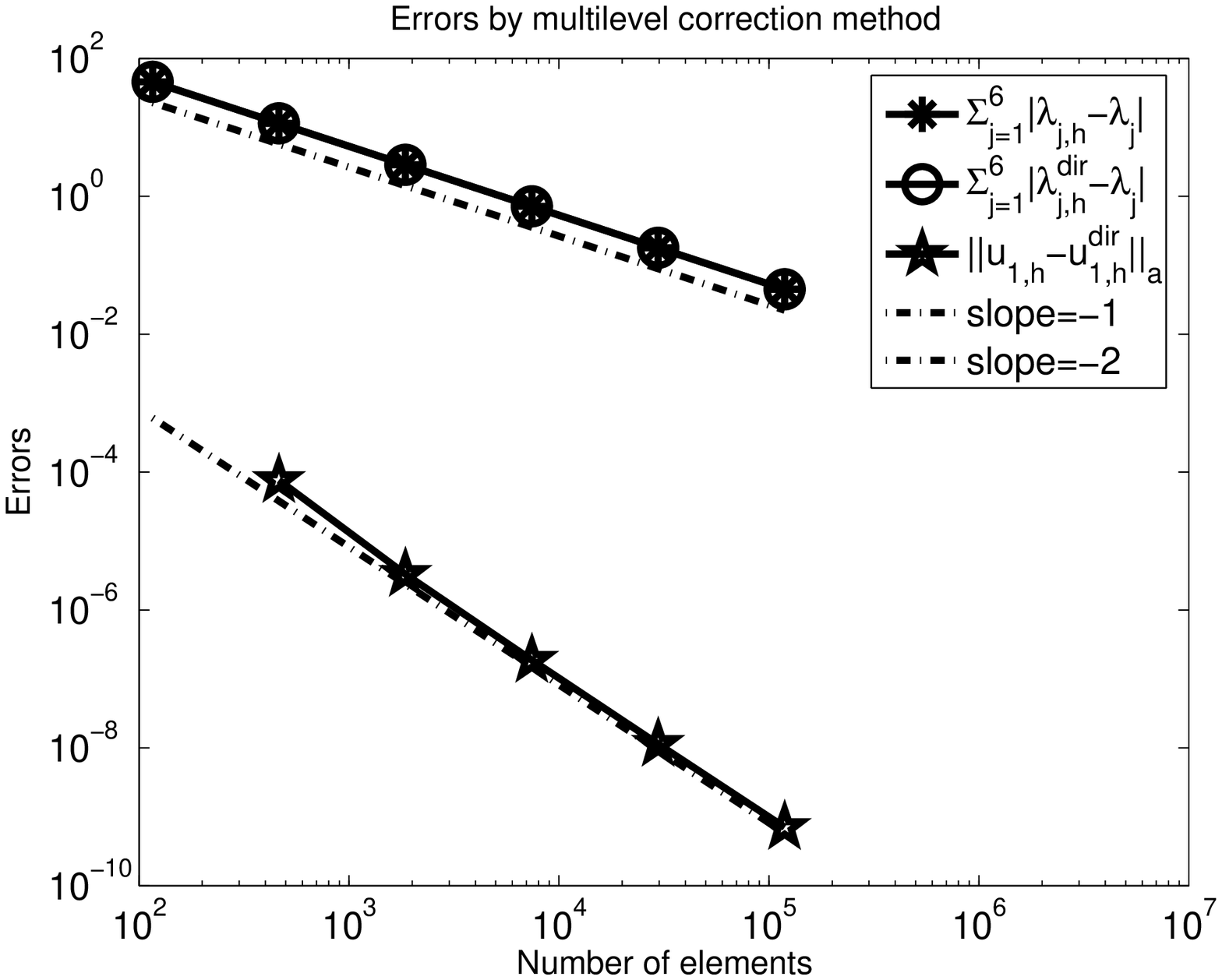}
\includegraphics[width=6.5cm,height=6cm]{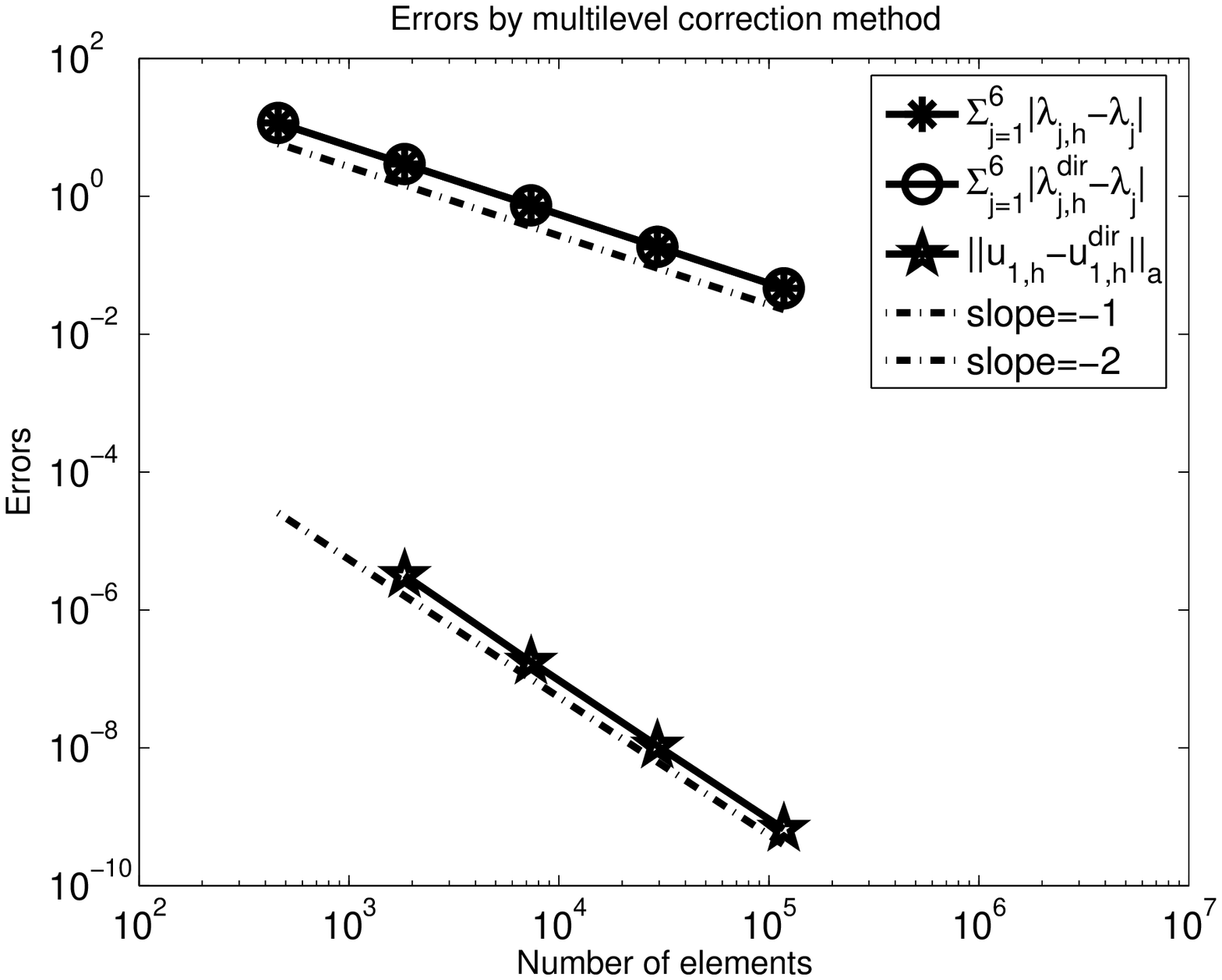}
\caption{\small\texttt The errors of the multilevel iteration
algorithm for the first six eigenvalues and the corresponding first eigenfunction,
where $u_h^{\mbox{dir}}$ and $\lambda_h^{\mbox{dir}}$ denote the eigenfunction
and eigenvalue approximation by direct eigenvalue solving (The left subfigure
is for the coarse initial mesh in the left of Figure \ref{Initial_Mesh}
and the right one for the fine initial mesh in the right of
 Figure \ref{Initial_Mesh})}\label{numerical_multi_grid_Exam_2}
\end{figure}

\section{Concluding remarks}
In this paper, we propose a type of multilevel method for eigenvalue problems based on the Newton iteration
scheme. In this type of iteration method, solving eigenvalue problem on the finest
finite element space is decomposed into solving a small scale eigenvalue problem in a coarse initial space and
solving a sequence of augmented linear problems, derived by Newton iteration step in the corresponding sequence
of finite element spaces. The proposed scheme improves the overall efficiency of eigenvalue problem solving
by the finite element method.

The quadratic convergence property of Newton's method improves the accuracy of the numerical solution.
On the other hand, the multilevel technique overcomes the sensitivity of initial guess of Newton scheme.

\bibliography{fullbib}

\end{document}